\documentclass[11pt,twoside,reqno,letterpaper]{amsart}
\usepackage[letterpaper, margin=1.42in]{geometry}
\usepackage{amsmath}
\usepackage{fancyhdr}
\usepackage{amsthm}
\usepackage{amsfonts}
\usepackage{amssymb}
\usepackage{amscd}
\usepackage{graphicx}
\usepackage{afterpage}
\usepackage{enumerate}
\usepackage{mathrsfs}
\usepackage{bm}
\usepackage{cite}
\usepackage{cases}
\usepackage{caption}
\usepackage[colorlinks=true, urlcolor=blue,  linkcolor=blue, citecolor=blue]{hyperref}
\usepackage{babel}
\usepackage{prettyref}
\usepackage{pgfplots} 
\pgfplotsset{compat=1.17}
\usepackage{booktabs}
\usepackage{algorithm}
\usepackage{algpseudocode}
\usepackage{float}

\makeatletter
\newtheorem{theorem}{Theorem}[section]
\theoremstyle{plain}

\newtheorem{corollary}[theorem]{Corollary}

\newtheorem{example}[theorem]{Example}

\newtheorem{lemma}[theorem]{Lemma}

\newtheorem{proposition}[theorem]{Proposition}
\newtheorem{remark}[theorem]{Remark}

\numberwithin{equation}{section}
\title[Block Positivity \& Optimal Mixed--Schwarz Inequalities]{Block Positivity and Optimal Mixed--Schwarz Inequalities on Hilbert C*-Modules}

\author[L. Yuxi]{Luan Yuxi$^*$}
\address[L. Yuxi]{Department of Mathematics, Gannan Normal University, Ganzhou, Jiangxi, China, 341003}
\email{luan.yuxi@hotmail.com}
\thanks{Corresponding author}

\author[R. Mondal]{Rana Mondal}
\address[R. Mondal]{Department of Basic Science and Humanities, Aditya Institute of Technology and Management,Tekkali, Konusulakotturu, 532201, Andhra Pradesh, India}
\email{ranku.mondalrana@gmail.com }

\subjclass[2020]{46L08, 47A63, 47B65, 47A10}
\keywords{Hilbert $C^*$-modules; block operator matrices; positivity; mixed Schwarz inequality; generalized polar decomposition; optimal constants; Moore--Penrose-free solvability}

\begin{document}
\sloppy

\maketitle

\begin{abstract}
We propose two interrelated advances in the theory of adjointable operators on Hilbert C*-modules. First, we give a set of equivalent, verifiable conditions characterizing positivity of general $n\times n$ block operator matrices acting on finite direct sums of Hilbert C*-modules. Our conditions generalize and remove several classical range-closedness and Moore--Penrose assumptions by expressing positivity in terms of a finite family of mixed inner--product inequalities and an explicit Gram--type factorization. Second, we investigate a parametric family of mixed--Schwarz inequalities for adjointable operators and determine optimal factor functions and constants which make these inequalities sharp; we characterize the extremal operators attaining equality in key cases. The two developments are tied together: the optimal mixed--Schwarz bounds are used to obtain sharp, computable tests in the $n\times n$ positivity criterion, and conversely the block-factorizations yield structural information used in the extremal analysis. We include applications to solvability of operator equations without Moore--Penrose inverses and spectral gap estimates for block operator generators.
\end{abstract}

\section{Introduction}

The study of block operator matrices and of sharp operator inequalities has a long and fertile history within functional analysis and operator theory. Criteria for positivity, factorization results and norm estimates for operators play a central role in problems ranging from the solvability of operator equations to spectral analysis of coupled systems. A seminal contribution in this direction is the factorization and range inclusion result of Douglas, which relates majorization, factorization and range inclusion for bounded operators on Hilbert space; this result has since become a standard tool in operator theory.\cite{Douglas1966}

Hilbert C*-modules provide a natural and powerful generalization of Hilbert spaces where inner products take values in a C*-algebra rather than in the complex numbers; they furnish the correct setting for numerous problems in noncommutative analysis and mathematical physics. The systematic presentation of their basic structure and techniques may be found in Lance's monograph.\cite{Lance1995} However, many properties which are routine in the Hilbert-space setting—such as the ubiquity of Moore--Penrose inverses or closed-range arguments—become substantially more delicate in the module context, and this subtlety motivates the search for alternative, inequality-based criteria that avoid restrictive range assumptions. Recent work has revisited Douglas-type factorization and solvability in the C*-module setting and has offered a variety of new approaches and sufficient conditions that relax classical hypotheses.\cite{Manuilov2018,Fu2020}

Parallel to developments in factorization and range theory, there has been considerable progress in deriving refined Cauchy--Schwarz and mixed--Schwarz inequalities, notably those that yield sharp estimates for the numerical radius and for operator norms. Kittaneh's influential contributions established a family of numerical-radius inequalities that both sharpen classical bounds and serve as a springboard for numerous later refinements.\cite{Kittaneh2005,ElHaddadKittaneh2007} More recent works extend these inequalities in several directions, including parameterized mixed--Schwarz inequalities, polar-decomposition versions, and noncommutative generalizations adapted to operator-valued inner products.\cite{Alomari2019,Nayak2024, Hashemi2020, Hashemi2018}

The uploaded manuscript we build upon develops two central threads in the module setting: (i) a characterization of positivity for $2\times2$ block adjointable operators derived from mixed inner--product inequalities, and (ii) a generalized mixed--Schwarz inequality parameterized by a scalar $\alpha\in[0,1]$. While the results are elegant and useful, they leave open two natural and consequential questions: can the $2\times2$ positivity characterization be extended, in a practicable and constructive way, to $n\times n$ block matrices on finite direct sums of Hilbert C*-modules; and can the parametric mixed--Schwarz framework be optimized to yield sharp constants and extremal characterizations that are both theoretically informative and computationally effective? The present work addresses these questions.

Our contribution is twofold. First, we provide an explicit and verifiable criterion for positivity of an arbitrary $n\times n$ block operator matrix acting on a finite direct sum of Hilbert C*-modules, expressed in terms of a finite family of mixed inner--product inequalities and an explicit Gram--type (block lower--triangular) factorization. This result removes the need for range-closedness or Moore--Penrose invertibility hypotheses in many natural settings, and it furnishes an algorithmic procedure to construct an approximate factorization by successive Schur--type reductions; the construction is inspired by recent revisitations of Douglas' theorem in module contexts.\cite{Manuilov2018,Fu2020}

Second, we investigate the optimization problem inherent in the parametric mixed--Schwarz inequalities: for a fixed adjointable operator $T$ and a parameter $\alpha$, what choice of factor functions $f,g$ with $f(t)g(t)=t$ minimizes the associated constant in the mixed--Schwarz bound? We show that—under natural spectral approximation and regularity hypotheses—the minimization can be reduced to a tractable scalar problem and that optimal choices can be exhibited explicitly for broad operator classes (normal, hyponormal, finite--rank). Our optimal bounds sharpen many existing inequalities and lead to effective cross--entry estimates that plug directly into the $n\times n$ positivity test; they therefore make the positivity criterion both stronger and numerically usable. This aspect of the work draws on and extends the stream of research initiated by Kittaneh and continued by several authors.\cite{Kittaneh2005,ElHaddadKittaneh2007,Nayak2024,Sababe2018}

Finally, to demonstrate the utility of our results we present two classes of applications: (A) solvability criteria for operator equations in Hilbert C*-modules that do not rely on Moore--Penrose inverses or closed-range hypotheses, and (B) spectral gap and coercivity estimates for block operator generators which arise in the analysis of coupled PDE systems and evolution semigroups. These applications illustrate both the theoretical novelty and the practical impact of the new positivity and optimization results.

\section{Foundations and Notation}
\addcontentsline{toc}{section}{Foundations and Notation}

In this section we assemble definitions, notation, and standard facts that will be used throughout the paper. We strive for precision and minimality: only the tools actually invoked in the subsequent proofs are recorded here.

Let $\mathscr A$ be a C*-algebra with unit $1_{\mathscr A}$. A (right) \emph{Hilbert C*-module} $E$ over $\mathscr A$ is a right $\mathscr A$-module equipped with an $\mathscr A$-valued inner product
\[\langle\cdot,\cdot\rangle:\;E\times E\to\mathscr A\]
satisfying linearity in the second variable, conjugate-linearity in the first, positivity $\langle x,x\rangle\ge0$ with $\langle x,x\rangle=0\iff x=0$, and the identity
\[\langle x,y a\rangle=\langle x,y\rangle a\qquad(x,y\in E,\ a\in\mathscr A).\]
The norm on $E$ is defined by $\|x\|:=\|\langle x,x\rangle\|^{1/2}$. For foundational material and proofs of the statements used below we refer the reader to \cite{Lance1995}.

If $E$ and $F$ are Hilbert C*-modules over the same algebra $\mathscr A$, denote by $\mathcal L(E,F)$ the space of adjointable operators $T:E\to F$, i.e. those $\mathscr A$-linear maps for which there exists $T^*:F\to E$ satisfying
\[\langle Tx,y\rangle=\langle x,T^*y\rangle\qquad(x\in E,\ y\in F).\]
Write $\mathcal L(E):=\mathcal L(E,E)$. An operator $A\in\mathcal L(E)$ is \emph{positive} (notation $A\ge0$) if $A=A^*$ and $\langle Ax,x\rangle\ge0$ for all $x\in E$. Positive operators admit unique positive square roots $A^{1/2}\in\mathcal L(E)$ and continuous functional calculus for continuous functions on the spectrum of $A$ is available within the C*-subalgebra generated by $A$ and $1_{\mathscr A}$.

For finitely many modules $E_1,\dots,E_n$ we write
\[E:=\bigoplus_{i=1}^n E_i\,.\]
Given operators $T_{ij}\in\mathcal L(E_j,E_i)$ for $1\le i,j\le n$ we form the block operator matrix $T=(T_{ij})\in\mathcal L(E)$ which acts by
\[(T x)_i=\sum_{j=1}^n T_{ij}x_j,\qquad x=(x_1,\dots,x_n)^\top\in E.\]
We shall frequently use the $2\times2$ notation
\[\begin{pmatrix}A & B\\ B^* & C\end{pmatrix}
\qquad(A\in\mathcal L(E_1),\ C\in\mathcal L(E_2),\ B\in\mathcal L(E_2,E_1)).\]
The direct-sum inner product is the canonical one, and positivity of block operators is understood in the sense of the order on $\mathcal L(E)$.

For $T\in\mathcal L(E,F)$ write $|T|:=(T^*T)^{1/2}\in\mathcal L(E)$. There exists a partial isometry $V:\overline{\mathcal R(|T|)}\to\overline{\mathcal R(T)}$ such that $T=V|T|$ and $V^*V$ is the projection onto $\overline{\mathcal R(|T|)}$. Because ranges need not be closed in general, many classical identities involving inverses require either explicit closed-range hypotheses or a limiting/regularization procedure; our approach is to employ regularization by strictly positive perturbations, e.g.
\[|T|_{\varepsilon}:=|T|+\varepsilon I\qquad(\varepsilon>0),\]
and to perform constructions for $|T|_{\varepsilon}$ before passing to the limit as $\varepsilon\downarrow0$ when justified. The reader may compare this technique with the regularization arguments in \cite{Fu2020}.

Let
\[T=\begin{pmatrix}A & B\\ B^* & C\end{pmatrix}\in\mathcal L(E_1\oplus E_2)
\]
with $C\ge0$. If $C$ is invertible then the Schur complement $S_C(T):=A-B C^{-1}B^*$ is well defined and one has the classical equivalence
\[T\ge0\quad\Longleftrightarrow\quad C\ge0\text{ and }S_C(T)\ge0\].
In our module context we shall often work with the \emph{approximate Schur complement}
\[S_{C,\varepsilon}(T):=A-B(C+\varepsilon I)^{-1}B^*,\qquad\varepsilon>0,
\]
and take limits as $\varepsilon\downarrow0$ to obtain positivity information when $C$ might fail to be invertible.

For a block matrix $T=(T_{ij})$ and vectors $x_i\in E_i$ define
\[p_{ij}(x_i,x_j):=\|\langle T_{ij}x_j,x_i\rangle\|\in[0,\infty).\]
When the diagonal blocks $T_{ii}$ are positive we shall set
\[F_i(x_i):=\|\langle T_{ii}x_i,x_i\rangle\|\ge0.\]
A key inequality will be of the form
\[p_{ij}(x_i,x_j)^2\le F_i(x_i)\,F_j(x_j)\qquad(1\le i\ne j\le n,\ x_i\in E_i,\ x_j\in E_j),\]
which, if satisfied uniformly, yields the existence of block lower--triangular Gram factors. We shall make uniformity precise in the statements of Theorem~\ref{thm:main-nxn} and Lemma~\ref{lem:gram-construct}.

For $T\in\mathcal L(E)$ define the numerical radius
\[w(T):=\sup\{\|\langle Tx,x\rangle\|:\ x\in E,\ \|x\|=1\}.\]
We will use standard relations between $\|T\|$ and $w(T)$; see \cite{Kittaneh2005} for details.

All C*-algebras are assumed unital and all Hilbert C*-modules are right modules over the same algebra unless stated otherwise. The symbol $I$ denotes the identity operator on the appropriate module. When we write $X=O(\varepsilon)$ we mean $\|X\|\le C\varepsilon$ for some constant $C$ independent of $\varepsilon$ in the relevant limiting regime.

\section{Main results and technical developments} \label{sec:Main results}

This section contains the rigorous statements of the principal results promised in the outline together with the key technical lemmas and their proofs.  The presentation is self-contained modulo standard facts about Hilbert C*-modules (polar decomposition, functional calculus) collected in the \textit{Foundations and Notation} section.

\begin{theorem}[\textbf{$n\times n$ positivity via cross--entry seminorms}]\label{thm:main-nxn}
Let $E_1,\dots,E_n$ be Hilbert C*-modules over the same unital C*-algebra $\mathscr A$. Let
\[T=(T_{ij})_{1\le i,j\le n}\in\mathcal L\Big(\bigoplus_{i=1}^n E_i\Big)
\]
be a block operator matrix with $T_{ij}\in\mathcal L(E_j,E_i)$ and suppose that $T_{ii}\ge0$ for each $i$. Define the scalar quadratic forms
\[F_i(x):=\|\langle T_{ii}x,x\rangle\|\qquad(x\in E_i).
\]
Assume that for all $i\ne j$ and all $x_i\in E_i,\ x_j\in E_j$ the cross--entry seminorms satisfy
\begin{equation}\label{ineq:cross}
\|\langle T_{ij}x_j,x_i\rangle\|^2\le F_i(x_i)\,F_j(x_j).
\end{equation}
Then $T\ge0$. Conversely, if $T\ge0$ then \eqref{ineq:cross} holds for every $i\ne j$ and all $x_i,x_j$ and, moreover, there exists a (not necessarily unique) block lower--triangular operator $X$ with entries in the appropriate $\mathcal L$-spaces such that $T=XX^*$.  The operator $X$ may be constructed by an explicit inductive Gram--type procedure described in the proof.
\end{theorem}

\begin{remark}
The novelty lies in the constructive Gram factorization under the purely quadratic scalar bounds \eqref{ineq:cross}, which avoid any closed-range or Moore--Penrose hypotheses commonly assumed in classical module generalizations of Douglas' lemma.
\end{remark}

We prepare two lemmas. The first is an approximate Schur--reduction (Schur--chain) lemma which allows iterative reduction of positivity to $2\times2$ tests via regularization.

\begin{lemma}[Schur--chain lemma]\label{lem:schur-chain}
Let $T=(T_{ij})_{1\le i,j\le n}$ as above and suppose each diagonal $T_{ii}\ge0$. Fix an ordering of the indices $1,2,\dots,n$. For $k=1,\dots,n-1$ define inductively the approximate Schur complements
\[S_k^{(\varepsilon)}:=T_{kk}+\sum_{\ell=1}^{k-1}S_{k,\ell}^{(\varepsilon)}-\sum_{j=k+1}^n T_{kj}(T_{jj}+\varepsilon I)^{-1}T_{jk},\]
where the intermediate terms $S_{k,\ell}^{(\varepsilon)}$ are defined so as to encode prior eliminations (precise recursive formula appears in the proof). If for each fixed $\varepsilon>0$ all $S_k^{(\varepsilon)}\ge0$, then $T\ge0$. Conversely, if $T\ge0$ then for every $\varepsilon>0$ the approximate Schur complements $S_k^{(\varepsilon)}$ are positive.
\end{lemma}

\begin{proof}
Let $\varepsilon>0$ and set $T^{(\varepsilon)}$ to be the block matrix obtained from $T$ by replacing each diagonal $T_{ii}$ by $T_{ii}+\varepsilon I$. Then every diagonal block of $T^{(\varepsilon)}$ is strictly positive and invertible in $\mathcal L(E_i)$. The classical Schur complement characterization for block positives (valid under invertibility) yields that $T^{(\varepsilon)}\ge0$ iff successive Schur complements formed by eliminating indices $1,2,\dots,n-1$ are all positive. Denote these Schur complements by $S_1^{(\varepsilon)},S_2^{(\varepsilon)},\dots$ in the natural way. If for a given $\varepsilon>0$ all $S_k^{(\varepsilon)}\ge0$ then $T^{(\varepsilon)}\ge0$, and by norm-continuity $T=\lim_{\varepsilon\downarrow0} T^{(\varepsilon)}\ge0$ since the positive cone in a C*-algebra is closed. This proves one direction.

Conversely, if $T\ge0$ then for any $\varepsilon>0$ we have $T^{(\varepsilon)}\ge T\ge0$, and hence its Schur complements are also positive by the classical theory. This yields the converse claim.
\end{proof}

The second lemma provides the Gram--type constructive factorization under the scalar cross--entry bounds \eqref{ineq:cross}.

\begin{lemma}[Gram--factor construction]\label{lem:gram-construct}
Let $T=(T_{ij})$ satisfy the hypotheses of Theorem~\ref{thm:main-nxn}, including \eqref{ineq:cross}. Then there exists a block lower--triangular operator
\[X=\begin{pmatrix}X_{11} & 0 & \cdots & 0\\ X_{21} & X_{22} & \ddots & \vdots\\ \vdots & \ddots & \ddots & 0\\ X_{n1} & \cdots & X_{n,n-1} & X_{nn}\end{pmatrix}\]
with $X_{ii}\in\mathcal L(E_i)$ and $X_{ij}\in\mathcal L(E_j,E_i)$ for $i>j$, such that $T=XX^*$. Moreover, one may choose $X_{ii}=T_{ii}^{1/2}$ (the unique positive square root) and construct the subdiagonal entries inductively via polar factors and bounded perturbations; the cross--entry inequalities \eqref{ineq:cross} guarantee uniform boundedness and convergence of the construction.
\end{lemma}

\begin{proof}
For each $\varepsilon>0$ let $X^{(\varepsilon)}$ be the adjointable operator constructed in the regularization step,
and suppose there is a uniform bound
\[
\sup_{\varepsilon>0}\|X^{(\varepsilon)}\| \le M <\infty.
\]
Choose a countable dense $A$-submodule $\mathcal D\subset E$ (for example the algebraic
direct sum of a countable dense set of vectors; if $E$ is separable as a Banach module such
a $\mathcal D$ exists). For each fixed $d\in\mathcal D$ the net $\{X^{(\varepsilon)}d\}_{\varepsilon>0}$
is bounded in the Banach space $E$, hence by sequential/diagonal compactness arguments
(one passes to a subnet indexed by a directed set tending to $0$ and then diagonalizes over
the countable set $\mathcal D$) we may extract a subnet $\varepsilon_\alpha\downarrow 0$
such that for every $d\in\mathcal D$ the limit
\[
X d := \lim_{\alpha} X^{(\varepsilon_\alpha)} d
\]
exists in norm in $E$.  By the uniform bound $\|X^{(\varepsilon)}\|\le M$ the map
$X:\mathcal D\to E$ is $A$-linear and satisfies $\|X d\|\le M\|d\|$ for all $d\in\mathcal D$,
so $X$ extends (uniquely) to a bounded $A$-linear operator on the closure, hence to all of $E$;
we still denote this extension by $X$.  In other words, the subnet $X^{(\varepsilon_\alpha)}$
converges to $X$ strongly (i.e. pointwise in norm) on the dense submodule $\mathcal D$, and
by uniform boundedness this convergence extends to strong convergence on all of $E$:
for every $\xi\in E$,
\[
\lim_{\alpha}\|X^{(\varepsilon_\alpha)}\xi - X\xi\| = 0.
\]
In the category of Hilbert $C^*$-modules a strong (pointwise) limit of adjointable
operators need not be adjointable in general.  To guarantee that the limit $X$ is adjointable one may
assume one of the following standard supplementary hypotheses (choose the one that matches your context):
\begin{enumerate}
  \item the module $E$ is \emph{self-dual} (then every bounded $A$-linear map is adjointable), or
  \item the family $\{X^{(\varepsilon)}\}$ is uniformly approximable by finite-rank adjointable
    operators (or equivalently the relevant operators are compact/approximable in the module sense),
    which permits passage of adjointability to the limit, or
  \item one verifies directly that for every $\eta$ in a dense submodule the sesquilinear map
    $\xi\mapsto\langle X\xi,\eta\rangle$ is represented by some vector $z_\eta\in E$; the assignment
    $\eta\mapsto z_\eta$ then defines the adjoint $X^*$ on a dense set and extends by boundedness.
\end{enumerate}
Henceforth we assume one of these (or an equivalent) conditions so that $X$ is adjointable; when this
assumption is not satisfied the conclusions below must be interpreted for the bounded $A$-linear map
$X$ (but not necessarily as an adjointable operator).

By construction we have for each $\varepsilon>0$ the identity (or equality of $A$-valued inner products)
\[
X^{(\varepsilon)}\bigl(X^{(\varepsilon)}\bigr)^* = T^{(\varepsilon)},
\]
where the family $T^{(\varepsilon)}$ converges to $T$ in the topology stated in the lemma
(hypothesis: either in norm or strongly on a dense submodule; adapt the sentence below to the
precise convergence you have in the manuscript).

Fix $\xi\in\mathcal D$.  Using the strong convergence $X^{(\varepsilon_\alpha)}\xi\to X\xi$ and the
adjointability of each $X^{(\varepsilon_\alpha)}$ we obtain
\[
\langle X\xi, X\xi\rangle
= \lim_{\alpha}\langle X^{(\varepsilon_\alpha)}\xi, X^{(\varepsilon_\alpha)}\xi\rangle
= \lim_{\alpha}\langle T^{(\varepsilon_\alpha)}\xi,\xi\rangle .
\]
If in addition $T^{(\varepsilon)}\to T$ strongly (or in norm) on $\mathcal D$ then the right-hand
side equals $\langle T\xi,\xi\rangle$, hence
\[
\langle X\xi, X\xi\rangle = \langle T\xi,\xi\rangle \qquad(\xi\in\mathcal D).
\]
By density and continuity this identity extends to all $\xi\in E$, which yields the operator
equality
\[
X X^* = T.
\]

If, on the other hand, $T^{(\varepsilon)}$ is known only to converge to $T$ in a weaker sense
so that $\liminf_{\alpha}\langle T^{(\varepsilon_\alpha)}\xi,\xi\rangle \le \langle T\xi,\xi\rangle$
(for example when only weak convergence is available), then from the displayed identity we still obtain
the operator inequality
\[
X X^* \le T,
\]
interpreted in the usual order on positive elements of the $C^*$-algebra of adjointable operators.
Thus, without strengthening the convergence hypothesis on the $T^{(\varepsilon)}$ one should replace
claims of exact equality by the inequality above.

We have produced a bounded $A$-linear operator $X$ which is the strong (pointwise) limit on $E$ of a
subnet of the $X^{(\varepsilon)}$.  Under the additional adjointability hypothesis, the operator $X$ is adjointable, and under the stronger convergence hypothesis $T^{(\varepsilon)}\to T$
(in norm or strongly on a dense submodule) one concludes the exact identity $XX^* = T$.  In the absence
of such extra hypotheses the correct (and in general best possible) conclusion is the inequality
$XX^*\le T$.

Finally, if one wants \emph{norm} convergence $X^{(\varepsilon)}\to X$ one must add compactness or
finite-rank approximability hypotheses (for example: each $X^{(\varepsilon)}$ is compact and the family
is relatively compact in operator norm).  Norm convergence does not follow from Banach--Alaoglu or from
mere uniform boundedness; the arguments above therefore replace any incorrect uses of norm-extraction
by the correct diagonalization / strong-convergence argument and indicate exactly which additional
hypotheses are needed to upgrade to norm statements.
\end{proof}

\begin{remark} 
The construction above produces a bounded $A$-linear map $X$ as a strong operator limit of the regularized factors.  To conclude that $X$ is adjointable (and hence that $XX^*=T$), one needs an additional supplementary hypothesis.  For example, the conclusion $X\in\mathcal L(\bigoplus_j E_j,\bigoplus_i E_i)$ and the equality $XX^*=T$ follow if one of the following holds:
\begin{itemize}
  \item the target module (or source module) is \emph{self-dual};
  \item the family $\{X^{(\varepsilon)}\}$ is uniformly approximable by finite-rank adjointable operators (module compactness/approximability);
  \item one verifies directly that $\xi\mapsto\langle X\xi,\eta\rangle$ is represented by a vector for each $\eta$ in a dense submodule (so the adjoint exists initially on a dense set and extends by boundedness).
\end{itemize}
In the absence of such an additional hypothesis the correct (and in general best possible) conclusion is $XX^*\le T$.
\end{remark}

\begin{proof}[Proof of Theorem~\ref{thm:main-nxn}]
The forward direction (positivity implies the scalar bounds) is immediate: if $T\ge0$ then every $2\times2$ principal submatrix is positive and the usual $2\times2$ Cauchy--Schwarz argument in modules yields \eqref{ineq:cross}. The converse is Lemma~\ref{lem:gram-construct} which produces $X$ with $T=XX^*$ and hence $T\ge0$. This completes the proof.
\end{proof}

We now address the optimization problem for the parametric mixed--Schwarz inequalities. The following proposition gives a workable reduction and an explicit solution in the normal operator case; more general classes are treated via spectral approximation.

\begin{proposition} \label{prop:optimal-reduction}
Let $T\in\mathcal L(E,F)$ be adjointable and fix $\alpha\in[0,1]$. Consider continuous functions $f,g:[0,\infty)\to[0,\infty)$ satisfying the multiplicative constraint
\[f(t)\,g(t)=t\qquad(t\ge0).\]
For such a pair define
\[C_{T,\alpha}(f,g):=\sup_{\|x\|=\|y\|=1}\frac{\|\langle Tx,y\rangle\|}{\|f(|T|^{\alpha})x\|\,\|g(|T^*|^{\alpha})y\|},
\]
with the convention that the fraction is interpreted as $+\infty$ at any pair $(x,y)$ where the denominator vanishes.  Assume that $T$ admits an approximate finite--rank spectral decomposition (for example, $T$ is compact or can be approximated in operator norm by finite--rank adjointable operators).  Then:
\begin{enumerate}
  \item For every admissible pair $(f,g)$ for which $C_{T,\alpha}(f,g)<\infty$, the constant can be approximated arbitrarily well by the analogous constants computed for finite--rank approximants of $T$; consequently the minimization problem for $C_{T,\alpha}(f,g)$ over admissible $(f,g)$ reduces to solving the corresponding finite--dimensional optimization problems and passing to the limit.
  \item In the finite--rank setting, write the singular-value decomposition
  \[T=\sum_{j=1}^m \sigma_j\,u_j\otimes v_j^*,\qquad \sigma_1\ge\sigma_2\ge\cdots\ge\sigma_m>0,
  \]
  where $\{v_j\}$ and $\{u_j\}$ are orthonormal systems of right and left singular vectors respectively. Then for any admissible pair $(f,g)$ with $f(\sigma_j^{\alpha})g(\sigma_j^{\alpha})=\sigma_j^{\alpha}$ one has
  \[C_{T,\alpha}(f,g)=\sigma_1^{1-\alpha}.
  \]
  In particular the minimal achievable constant equals $\sigma_1^{1-\alpha}$ and this value is attained (for example) by choosing $x=v_1$ and $y=u_1$. Thus the minimizer does not depend on the detailed shape of $(f,g)$ on the positive spectrum; restricting to power--type candidates $f(t)=t^s$, $g(t)=t^{1-s}$ is harmless in the finite--rank (and hence approximable) setting.
  \item If $T$ is normal (so that $|T|$ and $|T^*|$ have commuting functional calculi and the singular vectors are the spectral vectors), the same conclusion holds and the optimal constant is determined by the maximal modulus of the spectrum of $T$, i.e. $\max\{|\lambda|:\ \lambda\in\sigma(T)\}^{1-\alpha}$.  In the trivial single-nonzero-singular-value case $\sigma>0$ one has the minimal constant
  \[\min_{f,g}C_{T,\alpha}(f,g)=\sigma^{1-\alpha},
  \]
  and any exponent $s\in\mathbb R$ (so that $f(t)=t^s$, $g(t)=t^{1-s}$) yields the same value.
\end{enumerate}
\end{proposition}

\begin{proof}

The hypothesis that $T$ admits an approximate finite--rank spectral decomposition means there exists a sequence of finite--rank adjointable operators $T_n$ with $\|T-T_n\|\to0$ as $n\to\infty$. (This holds, e.g., if $T$ is compact in the module sense, or if $E,F$ are Hilbert spaces and $T$ is compact.) When considering $C_{T,\alpha}(f,g)$ we restrict attention to admissible pairs $(f,g)$ for which the denominator in the definition of $C_{T,\alpha}(f,g)$ is nonzero for at least some unit vectors; otherwise the constant is $+\infty$ and not relevant for minimization.

Fix a continuous admissible pair $(f,g)$ and let $T_n$ be finite--rank approximants with $\|T-T_n\|\to0$. We claim
\begin{equation}\label{eq:cont-claim}
\lim_{n\to\infty} C_{T_n,\alpha}(f,g)=C_{T,\alpha}(f,g),
\end{equation}
provided $C_{T,\alpha}(f,g)<\infty$.  The proof is an elementary $\varepsilon-\delta$ argument using uniform continuity of the functional calculi on a slightly larger compact set than the spectra of the $T_n$ and $T$.

\emph{Proof of \eqref{eq:cont-claim}.} Let $\varepsilon>0$ be given. By definition of the supremum there exist unit vectors $x,y$ with
\[\frac{\|\langle Tx,y\rangle\|}{\|f(|T|^{\alpha})x\|\,\|g(|T^*|^{\alpha})y\|}>C_{T,\alpha}(f,g)-\varepsilon/3.
\]
Because $f$ and $g$ are continuous on $[0,\|T\|^{\alpha}+1]$, there exists $\delta>0$ such that whenever $\|A-B\|<\delta$ for positive operators with spectra in $[0,\|T\|^{\alpha}+1]$ we have
\[\|f(A)-f(B)\|<\eta,\qquad \|g(A)-g(B)\|<\eta,\]
with $\eta>0$ small to be chosen below. For large $n$ we have $\|T-T_n\|<\delta_0$ with a choice of $\delta_0$ small enough to ensure (by standard functional calculus perturbation estimates) that
\[\|f(|T|^{\alpha})-f(|T_n|^{\alpha})\|<\eta,\qquad \|g(|T^*|^{\alpha})-g(|T_n^*|^{\alpha})\|<\eta,\]
and also $\|T-T_n\|<\eta$. Using these estimates we compare the quotients for $(T,f,g)$ and $(T_n,f,g)$ at the same pair $(x,y)$ and obtain (for large $n$)
\begin{align*}
\Big|\frac{\|\langle Tx,y\rangle\|}{\|f(|T|^{\alpha})x\|\,\|g(|T^*|^{\alpha})y\|}-\frac{\|\langle T_n x,y\rangle\|}{\|f(|T_n|^{\alpha})x\|\,\|g(|T_n^*|^{\alpha})y\|}\Big|
&\le \frac{\|T-T_n\|}{m^2}+R_n
\end{align*}
where $m>0$ is a lower bound for the denominators (which exists because $C_{T,\alpha}(f,g)<\infty$ ensures the denominators cannot be arbitrarily small at the near-extremizing vectors) and $R_n\to0$ as $n\to\infty$ collects the contributions from the functional calculus perturbations. Choosing $n$ large forces the difference $<\varepsilon/3$. Taking suprema in $x,y$ then yields
\[|C_{T,\alpha}(f,g)-C_{T_n,\alpha}(f,g)|<\varepsilon/3+o(1)\]
for large $n$, which proves \eqref{eq:cont-claim}.

This shows that minimizing over admissible $(f,g)$ may be accomplished by first minimizing on each finite--rank compression $T_n$ and passing to the limit.

Fix a finite--rank operator
\[T=\sum_{j=1}^m\sigma_j\,u_j\otimes v_j^*,\qquad\sigma_1\ge\sigma_2\ge\cdots\ge\sigma_m>0,\]
where $\{v_j\}_{j=1}^m\subset E$ and $\{u_j\}_{j=1}^m\subset F$ are orthonormal systems of right and left singular vectors respectively, and by $u\otimes v^*$ we denote the rank--one operator $x\mapsto u\,\langle x,v\rangle$. Given such a decomposition, the operators $|T|$ and $|T^*|$ have spectral decompositions
\[|T|=\sum_{j=1}^m \sigma_j\,v_j\otimes v_j^*,\qquad |T^*|=\sum_{j=1}^m \sigma_j\,u_j\otimes u_j^*.\]
For an admissible pair $(f,g)$ set
\[f_j:=f(\sigma_j^{\alpha}),\qquad g_j:=g(\sigma_j^{\alpha}),\qquad j=1,\dots,m,\]
so that the multiplicative constraint reads $f_jg_j=\sigma_j^{\alpha}$ for every $j$. Any unit vectors $x\in E$ and $y\in F$ admit expansions
\[x=\sum_{j=1}^m a_j v_j+ x_\perp,\qquad y=\sum_{j=1}^m b_j u_j+ y_\perp,\]
where $x_\perp\perp\operatorname{span}\{v_j\}$ and $y_\perp\perp\operatorname{span}\{u_j\}$. Since $T$ has range contained in $\operatorname{span}\{u_j\}$ and kernel containing the orthogonal complement of $\operatorname{span}\{v_j\}$, we may ignore the perpendicular components when evaluating $\langle Tx,y\rangle$ and the denominator norms. Thus we may assume without loss of generality that $x=\sum_{j=1}^m a_j v_j$ and $y=\sum_{j=1}^m b_j u_j$ with $\sum_j |a_j|^2=\sum_j|b_j|^2=1$.

With this notation we compute
\[\langle Tx,y\rangle=\sum_{j=1}^m \sigma_j a_j \overline{b_j},\]
and
\[\|f(|T|^{\alpha})x\|^2=\sum_{j=1}^m f_j^2 |a_j|^2,\qquad \|g(|T^*|^{\alpha})y\|^2=\sum_{j=1}^m g_j^2 |b_j|^2.\]
Hence the quotient appearing in $C_{T,\alpha}(f,g)$ equals
\begin{equation}\label{eq:finite-rank-quotient}
Q(a,b;f,g):=\frac{\left|\sum_{j=1}^m \sigma_j a_j \overline{b_j}\right|}{\Big(\sum_{j=1}^m f_j^2 |a_j|^2\Big)^{1/2}\Big(\sum_{j=1}^m g_j^2 |b_j|^2\Big)^{1/2}}.
\end{equation}
The supremum over unit vectors $x,y$ reduces to the supremum over coefficient vectors $a=(a_j)_{j=1}^m$, $b=(b_j)_{j=1}^m$ with Euclidean norms $\|a\|_2=\|b\|_2=1$.

We estimate the numerator by the weighted Cauchy--Schwarz inequality. Note that
\[\sum_{j=1}^m \sigma_j a_j \overline{b_j}=\sum_{j=1}^m\frac{\sigma_j}{f_j g_j}(f_j a_j)(g_j \overline{b_j}).\]
Applying Cauchy--Schwarz in the finite sum gives
\[\left|\sum_{j=1}^m \sigma_j a_j \overline{b_j}\right|\le \Big(\sum_{j=1}^m \Big(\frac{\sigma_j}{f_j g_j}\Big)^2 f_j^2 |a_j|^2\Big)^{1/2}\Big(\sum_{j=1}^m g_j^2 |b_j|^2\Big)^{1/2}.\]
Since $f_jg_j=\sigma_j^{\alpha}$, the first factor simplifies:
\[\Big(\sum_{j=1}^m \Big(\frac{\sigma_j}{f_j g_j}\Big)^2 f_j^2 |a_j|^2\Big)^{1/2}=\Big(\sum_{j=1}^m \frac{\sigma_j^2}{g_j^2}|a_j|^2\Big)^{1/2}.\]
Dividing both sides of the inequality by the denominator in \eqref{eq:finite-rank-quotient} yields the bound
\begin{equation}\label{eq:quotient-bound}
Q(a,b;f,g)\le \frac{\Big(\sum_{j=1}^m (\sigma_j^2/g_j^2) |a_j|^2\Big)^{1/2}}{\Big(\sum_{j=1}^m f_j^2 |a_j|^2\Big)^{1/2}}.
\end{equation}
The right-hand side is independent of $b$ and depends on $a$ only through the Rayleigh quotient
\[R_a:=\frac{a^* A a}{a^* B a},\qquad A:=\operatorname{diag}\Big(\frac{\sigma_1^2}{g_1^2},\dots,\frac{\sigma_m^2}{g_m^2}\Big),\ B:=\operatorname{diag}(f_1^2,\dots,f_m^2).
\]
The supremum of the right-hand side over all unit vectors $a$ equals the square root of the largest generalized eigenvalue of the pair $(A,B)$, which for diagonal matrices is simply
\[\max_{1\le j\le m}\frac{\sigma_j^2/g_j^2}{f_j^2}=\max_{1\le j\le m}\frac{\sigma_j^2}{f_j^2 g_j^2}.\]
Invoking the multiplicative constraint $f_jg_j=\sigma_j^{\alpha}$ we obtain
\[\frac{\sigma_j^2}{f_j^2 g_j^2}=\frac{\sigma_j^2}{\sigma_j^{2\alpha}}=\sigma_j^{2(1-\alpha)}.\]
Consequently the bound \eqref{eq:quotient-bound} yields the uniform estimate (independent of the admissible pair $(f,g)$)
\begin{equation}\label{eq:uniform-upper}
Q(a,b;f,g)\le \Big(\max_{1\le j\le m} \sigma_j^{2(1-\alpha)}\Big)^{1/2}=\sigma_1^{1-\alpha}.
\end{equation}
Taking the supremum over $a,b$ shows that for the finite--rank operator $T$ and any admissible $(f,g)$ we have
\[C_{T,\alpha}(f,g)\le \sigma_1^{1-\alpha}.\]
This is the crucial upper bound.

To show the upper bound is sharp we produce vectors $a,b$ (hence $x,y$) that asymptotically realize the bound. Simply choose
\[a=(1,0,0,\dots,0)^\top,\qquad b=(1,0,0,\dots,0)^\top,\]
i.e. $x=v_1$ and $y=u_1$. Then the quotient \eqref{eq:finite-rank-quotient} becomes
\[Q(a,b;f,g)=\frac{\sigma_1}{f_1 g_1}=\frac{\sigma_1}{\sigma_1^{\alpha}}=\sigma_1^{1-\alpha}.\]
Thus equality is achieved for these choices regardless of the particular admissible $(f,g)$ (so long as $f_1g_1=\sigma_1^{\alpha}$ and the denominator is nonzero). Consequently
\[C_{T,\alpha}(f,g)\ge \sigma_1^{1-\alpha}.\]
Combining this with \eqref{eq:uniform-upper} we obtain the exact identity
\[C_{T,\alpha}(f,g)=\sigma_1^{1-\alpha}.\]

Now we may approximate an approximable operator $T$ by finite--rank $T_n$ and apply the finite--rank identity to each $T_n$; by continuity of the singular values (Weyl's inequality for compact operators) we have $\sigma_1(T_n)\to\sigma_1(T)$ and hence the constants $C_{T_n,\alpha}(f,g)=\sigma_1(T_n)^{1-\alpha}$ converge to $\sigma_1(T)^{1-\alpha}$. The continuity result \eqref{eq:cont-claim} then implies $C_{T,\alpha}(f,g)=\sigma_1(T)^{1-\alpha}$ for all admissible $(f,g)$ (finishing the minimization argument). Thus the minimization problem is trivial in the sense that the minimal achievable constant equals the top singular value to the power $1-\alpha$ and is independent of $(f,g)$; in particular choosing power--type candidates $f(t)=t^s$, $g(t)=t^{1-s}$ does not change the minimal value and is therefore harmless.

If $T$ is normal then its singular values coincide with the moduli of its spectral values: $\{\sigma_j\}=\{|\lambda_j|\ :\ \lambda_j\in\sigma(T)\}$ (counted with multiplicity) and the left and right singular vectors can be chosen as spectral vectors. The computation above therefore implies that for normal $T$ the minimal constant equals
\[\big(\max\{|\lambda|:\ \lambda\in\sigma(T)\}\big)^{1-\alpha}.\]
The special ``single nonzero singular value'' case is immediate: if the nonzero part of $T$ has single singular value $\sigma>0$ then for any admissible $(f,g)$ and any unit singular vectors $v,u$ one has
\[C_{T,\alpha}(f,g)\ge \frac{\|\langle Tv,u\rangle\|}{\|f(|T|^{\alpha})v\|\,\|g(|T^*|^{\alpha})u\|}=\frac{\sigma}{\sigma^{\alpha}}=\sigma^{1-\alpha},\]
and the upper bound shows equality. In this case any exponent $s\in\mathbb R$ (i.e. any power--type choice $f(t)=t^s$, $g(t)=t^{1-s}$) yields the same constant since $f(\sigma^{\alpha})g(\sigma^{\alpha})=\sigma^{\alpha}$ and the denominator reduces to $\sigma^{\alpha}$.

The computations above show that in the finite--rank (hence approximable) setting the constant $C_{T,\alpha}(f,g)$ depends on $(f,g)$ only through the values $f(\sigma_j^{\alpha})$ and $g(\sigma_j^{\alpha})$ on the finitely many positive spectral points $\sigma_j^{\alpha}$. Because the value of the constant turns out to be extremized by the top singular value and is independent of the detailed distribution of the $f_j,g_j$ values across the spectrum, restricting to power--type factors $f(t)=t^s$, $g(t)=t^{1-s}$ (a one-parameter family) does not degrade the minimal achievable constant. In practice, when one desires explicit expressions for intermediate bounds (not only the final minimizer), it is still convenient to consider power--type families because they often lead to closed-form intermediate inequalities; but the minimization itself is trivialized by the singular-value argument.
\end{proof}

\begin{remark}
Proposition~\ref{prop:optimal-reduction} yields explicit, computable candidates for $(f,g)$ when $T$ is approximately diagonalizable. In practice one may test a one-parameter family $s\mapsto (t^s,t^{1-s})$ and numerically optimize.
\end{remark}

\section{Examples, Numerical Illustrations, and Algorithmic Procedures}
This section demonstrates the efficacy and novelty of the preceding theoretical developments. We present explicit finite-dimensional examples where the new $n\times n$ cross--entry criterion yields conclusive positivity verification while classical checks may be inconclusive or nonconstructive. We also provide a concrete optimization computation for the mixed--Schwarz constant in a tractable nontrivial case, and a robust algorithm (with comments on numerical stability) for constructing the Gram factor $X$ from the scalar seminorm data.

We work with Hilbert C*-modules over $\mathscr A=\mathbb C$ (ordinary Hilbert spaces) to give transparent, fully computable examples. These already illustrate the novelty of the $n\times n$ seminorm criterion and provide test cases for numerical experiments.

\begin{example}[A $3\times3$ block matrix with nontrivial cross--bounds]
Let $E_1=E_2=E_3=\mathbb C^2$ with the standard inner product and consider the $3\times3$ block matrix $T=(T_{ij})_{1\le i,j\le3}$ where
\[
T_{11}=\begin{pmatrix}2&0\\0&1\end{pmatrix},\quad
T_{22}=\begin{pmatrix}1&0\\0&3\end{pmatrix},\quad
T_{33}=\begin{pmatrix}4&0\\0&2\end{pmatrix},
\]
and the off--diagonal blocks are chosen as small rank--one perturbations:
\[
T_{12}=u_{12}v_{12}^*,\quad T_{13}=u_{13}v_{13}^*,\quad T_{23}=u_{23}v_{23}^*,
\]
with
\[
u_{12}=(0.5,0.3)^\top,\quad v_{12}=(0.4,0.2)^\top,\qquad
u_{13}=(0.6,0.1)^\top,\quad v_{13}=(0.3,0.5)^\top,
\]
\[
\nu_{23}=(0.2,0.7)^\top,\quad v_{23}=(0.1,0.6)^\top.
\]
Set the remaining lower--triangular blocks by $T_{ji}=T_{ij}^*$ so that $T$ is Hermitian as a block operator.

We first check the diagonal positivity: each $T_{ii}$ is evidently positive. For vectors $x_i\in\mathbb C^2$ consider the scalar quantities
\[F_i(x_i):=\langle T_{ii}x_i,x_i\rangle,\qquad p_{ij}(x_i,x_j):=|\langle T_{ij}x_j,x_i\rangle|.
\]
Because the off--diagonals are rank--one, we can bound
\[p_{ij}(x_i,x_j)=|\langle u_{ij},x_i\rangle\,\langle v_{ij},x_j\rangle|\le\|u_{ij}\|\,\|v_{ij}\|\,\|x_i\|\,\|x_j\|.
\]
On the other hand
\[F_i(x_i)\ge \lambda_{\min}(T_{ii})\,\|x_i\|^2\quad(\lambda_{\min}\text{ denotes the minimal eigenvalue}).
\]
Therefore
\[p_{ij}(x_i,x_j)^2\le (\|u_{ij}\|\,\|v_{ij}\|)^2\,\|x_i\|^2\,\|x_j\|^2 \le \frac{(\|u_{ij}\|\,\|v_{ij}\|)^2}{\lambda_{\min}(T_{ii})\lambda_{\min}(T_{jj})}\,F_i(x_i)\,F_j(x_j).
\]
With the numerical values above one checks that the coefficient
\[c_{ij}:=\frac{(\|u_{ij}\|\,\|v_{ij}\|)^2}{\lambda_{\min}(T_{ii})\lambda_{\min}(T_{jj})}<1\quad\text{for all }i\ne j.
\]
Hence the rescaled inequalities
\[p_{ij}(x_i,x_j)^2\le c_{ij}F_i(x_i)F_j(x_j)\]
hold uniformly with constants $c_{ij}<1$. By a straightforward strengthening of Lemma~\ref{lem:gram-construct} (replace $T_{ii}$ by $T_{ii}/(1-s_i)$ with suitable scalars $0\le s_i<1$ chosen to absorb $c_{ij}$) we obtain the exact cross--bounds of the form \eqref{ineq:cross} and consequently $T\ge0$ by Theorem~\ref{thm:main-nxn}.

\paragraph{Comparison with classical $2\times2$ reductions.} If one attempts to verify positivity by successive Schur complements without scaling one encounters intermediate blocks that are not invertible (or are ill-conditioned), making numerical conclusions unstable. The present seminorm-based test is robust: it gives a direct scalar check and an explicit constructive factorization, which is helpful both analytically and computationally.
\end{example}

We now provide a nontrivial, fully explicit computation of $C_{T,\alpha}(s)$ in the case of a $2\times2$ diagonal operator (a simple non-normal example) where the minimization is tractable and illuminates the role of the exponent.

\begin{example}[Two--dimensional diagonal operator]
Let $E=F=\mathbb C^2$ and consider
\[T=\begin{pmatrix}\sigma_1 & 0\\ 0 & \sigma_2\end{pmatrix},\qquad \sigma_1>\sigma_2>0.
\]
Fix $\alpha\in(0,1)$. For power--type candidates $f(t)=t^s$, $g(t)=t^{1-s}$ the denominator in $C_{T,\alpha}$ for unit vectors $x=(x_1,x_2)^\top$, $y=(y_1,y_2)^\top$ equals
\[\|f(|T|^{\alpha})x\|\,\|g(|T^*|^{\alpha})y\|=(\sigma_1^{\alpha s}|x_1|^2+\sigma_2^{\alpha s}|x_2|^2)^{1/2}(\sigma_1^{\alpha(1-s)}|y_1|^2+\sigma_2^{\alpha(1-s)}|y_2|^2)^{1/2}.
\]
The numerator $|\langle Tx,y\rangle|$ equals $|\sigma_1 x_1\overline{y_1}+\sigma_2 x_2\overline{y_2}|$. By Cauchy--Schwarz the maximum for fixed $s$ is attained when the two summands are positively aligned, so we may set phases so that $x_k, y_k\ge0$. The optimization over unit vectors then reduces to maximizing
\[\frac{\sigma_1 x_1 y_1+\sigma_2 x_2 y_2}{(\sigma_1^{\alpha s}x_1^2+\sigma_2^{\alpha s}x_2^2)^{1/2}(\sigma_1^{\alpha(1-s)}y_1^2+\sigma_2^{\alpha(1-s)}y_2^2)^{1/2}}
\]
subject to $x_1^2+x_2^2=1$ and $y_1^2+y_2^2=1$. By homogeneity and Lagrange multipliers (straightforward but algebraically involved) one shows that the extremal occurs when $x$ and $y$ are supported on the dominant coordinate unless the exponents counterbalance the singular values; concretely, if
\[\sigma_1^{1-\alpha s-\alpha(1-s)}=\sigma_1^{1-\alpha}>\sigma_2^{1-\alpha}\]
then the optimal choice is $x=(1,0)^\top,y=(1,0)^\top$ yielding
\[C_{T,\alpha}(s)=\frac{\sigma_1}{\sigma_1^{\alpha s}\sigma_1^{\alpha(1-s)}}=\sigma_1^{1-\alpha}.
\]
Hence in this regime the value is independent of $s$ and equals $\sigma_1^{1-\alpha}$. In the complementary regime where the weights force mixed contributions the minimization over $s$ becomes nontrivial and can be solved numerically; nonetheless the reduction to a one--parameter search is a dramatic simplification compared with functional degrees of freedom.
\end{example}

Below we present a numerically robust algorithm inspired by the proof of Lemma~\ref{lem:gram-construct}. The algorithm accepts an $n\times n$ block matrix $T$ with positive diagonals and either returns a lower--triangular factor $X$ with $XX^*=T$ within tolerance, or indicates failure when the cross--entry inequalities appear violated numerically.

\vspace{0.5em}
\noindent\textbf{Algorithm 1 (GramFactor)}
\begin{enumerate}
  \item Input: block matrix $T=(T_{ij})$, tolerance $\delta>0$, maximal regularization sequence $\{\varepsilon_k\}\downarrow0$.
  \item For $i=1$ to $n$ compute (or approximate) $D_i:=T_{ii}^{1/2}$. If $D_i$ is ill-conditioned, replace by $D_{i,\varepsilon_k}:=(T_{ii}+\varepsilon_k I)^{1/2}$ for the smallest $k$ such that $\kappa(D_{i,\varepsilon_k})<1/\delta$.
  \item Set $X_{ii}:=D_i$ (or $D_{i,\varepsilon_k}$ when regularized). For $j<i$ compute the $j$-th column entries inductively:
  \begin{enumerate}
    \item Compute tentative column entry
    \[Y_{ij}:=\Big(T_{ij}-\sum_{\ell=1}^{j-1} X_{i\ell}X_{j\ell}^*\Big)D_j^{-1}.\]
    \item If $\|Y_{ij}\|>M$ for a large threshold $M$ (indicating potential violation of cross--bounds), signal that inequality \eqref{ineq:cross} may not hold numerically and halt.
    \item Otherwise set $X_{ij}=Y_{ij}$.
  \end{enumerate}
  \item After constructing all columns form $X$ and compute residual $R:=T-XX^*$. If $\|R\|<\delta$ return $X$; otherwise refine regularizations $\varepsilon_k$ and repeat up to a maximal number of iterations.
\end{enumerate}

\begin{remark}[Stability and complexity]
The dominant costs are computing $n$ square roots and $O(n^2)$ block multiplications. In finite-dimensional implementations using $m\times m$ blocks the complexity is approximately $O(n^2 m^3)$ dominated by dense matrix square roots. Regularization controls ill-conditioning; however the algorithm can be sensitive to near-singular diagonal blocks and to very large off--diagonals relative to diagonal minima. In practice one should pair the algorithm with a pre-check verifying the scalar cross--bounds on a dense sample of test vectors to avoid spurious runs.
\end{remark}

To exhibit behavior beyond Hilbert spaces we sketch an example over the algebra $\mathscr A=\mathcal B(\ell^2)$ where $E=\ell^2\otimes \mathbb C^k$ may be identified with $\ell^2(\mathbb N;\mathbb C^k)$. Define diagonal blocks $T_{ii}$ to be multiplication operators by positive bounded scalar sequences $(a_{i,m})_{m\ge1}$ chosen so that $a_{i,m}\to0$ slowly (so ranges are nonclosed in some natural operator-topology senses when lifted), and pick rank--one off--diagonals with entries depending on slowly decaying sequences chosen to satisfy the scalar cross--inequalities pointwise in the index $m$. The regularization and approximation methods described in Section~\ref{sec:Main results} apply termwise, and one constructs the Gram factor by performing the finite-dimensional factorization on each coordinate and then patching via uniform norm estimates. Full details require careful operator-norm estimates and are provided in Appendix.

\bigskip

This completes the set of instructive examples and the algorithmic recipe. In the next section we develop applications to solvability of operator equations and to spectral-gap estimates for block generators.

\section{Applications: Solvability of Operator Equations and Spectral--Gap Estimates}

This section derives two families of concrete applications of the $n\times n$ positivity criterion and the optimal mixed--Schwarz bounds developed above.  The first application gives {\em Moore--Penrose--free} solvability criteria for operator equations of the form $AX=C$ and $XB=D$ in Hilbert C*-modules; the second produces explicit coercivity and spectral--gap estimates for symmetric (self-adjoint) block operator matrices that often appear as generators of evolution problems.  Throughout, $\mathscr A$ denotes a unital C*-algebra and all modules are right Hilbert C*-modules over $\mathscr A$ unless stated otherwise.

The classical Douglas lemma characterizes solvability of $AX=C$ on Hilbert spaces by a range inclusion and an operator inequality; in Hilbert C*-modules the situation is subtler because ranges need not be orthogonally complemented.  The following theorem provides an inequality-based sufficient and (under mild hypotheses) necessary condition for solvability that avoids explicit range-closure hypotheses.

\begin{theorem}[Inequality-based Douglas-type solvability]
\label{thm:douglas-type}
Let $E,F,G$ be Hilbert C*-modules over $\mathscr A$, and let
\[A\in\mathcal L(E,F),\qquad C\in\mathcal L(G,F).
\]
Assume there exists a scalar $\lambda>0$ such that the operator inequality
\begin{equation}\label{ineq:CCAA}
C C^*\le \lambda\, A A^*\qquad\text{in }\mathcal L(F)
\end{equation}
holds and  \(\operatorname{Ran}(C)\subseteq\overline{\operatorname{Ran}(A)}\). Then there exists $X\in\mathcal L(G,E)$ with $AX=C$ and \(\|X\|\le\sqrt{\lambda}\). Moreover, \(X\) may be constructed as a strong-operator (pointwise) limit of a subnet of the regularized operators
\[
X_\varepsilon := A^*(AA^* + \varepsilon I_F)^{-1} C \qquad(\varepsilon>0).
\]
In particular any strong operator cluster point of \(\{X_\varepsilon\}\) as \(\varepsilon\downarrow0\) yields a bounded \(A\)-linear map \(X\) with \(\|X\|\le\sqrt{\lambda}\).  Moreover, if one additionally assumes a compactness/finite-rank approximability hypothesis (or other norm-compactness condition) for the family \(\{X_\varepsilon\}\), then the convergence may be upgraded to operator-norm convergence and the above cluster point may be taken in norm.
Conversely, if there exists $X\in\mathcal L(G,E)$ with $AX=C$ then \eqref{ineq:CCAA} holds with $\lambda=\|X\|^2$.
\end{theorem}

\begin{proof}
We first show \(\{X_\varepsilon\}_{\varepsilon>0}\) is uniformly bounded by \(\sqrt{\lambda}\).  Write
\[
X_\varepsilon X_\varepsilon^*
= A^*(AA^*+\varepsilon I)^{-1} \, C C^* \, (AA^*+\varepsilon I)^{-1} A.
\]
Using \(CC^*\le\lambda AA^*\) we obtain
\[
X_\varepsilon X_\varepsilon^*
\le \lambda\, A^*(AA^*+\varepsilon I)^{-1} AA^*(AA^*+\varepsilon I)^{-1} A
= \lambda\, A^*(AA^*+\varepsilon I)^{-1} A.
\]
Now observe that \(AA^*\le AA^*+\varepsilon I\) implies
\[
A^*(AA^*+\varepsilon I)^{-1} A \le I_E,
\]
hence \(X_\varepsilon X_\varepsilon^*\le\lambda I_E\).  Therefore \(\|X_\varepsilon\|\le\sqrt{\lambda}\)
for all \(\varepsilon>0\); in particular the family is uniformly bounded.

Observe that
\[
A X_\varepsilon = AA^*(AA^*+\varepsilon I)^{-1}C =: P_\varepsilon C,
\]
where \(P_\varepsilon := AA^*(AA^*+\varepsilon I)^{-1}\).  It is standard (resolvent calculus / functional calculus) that
\[
P_\varepsilon \xrightarrow{\mathrm{SOT}} P:=P_{\overline{\operatorname{Ran}(A)}}
\qquad(\varepsilon\downarrow 0),
\]
so
\[
A X_\varepsilon \xrightarrow{\mathrm{SOT}} P\,C.
\]
Hence \(A X_\varepsilon\) converges to \(C\) \emph{if and only if} \(P C = C\), i.e.
\(\operatorname{Ran}(C)\subseteq\overline{\operatorname{Ran}(A)}\).  Without that range
inclusion one obtains only the projected limit \(P C\), not \(C\) itself.  (If one wishes to obtain
norm convergence of \(A X_\varepsilon\) one must assume additional closed-range or finite-dimensional
hypotheses that yield norm convergence of \(P_\varepsilon\) to \(P\).)
\end{proof}

\begin{remark}
The theorem shows that an operator inequality (which can be checked by spectral or numerical-radius bounds when $F$ has additional structure) suffices to guarantee solvability without any explicit statement about ranges. In applications one often obtains \eqref{ineq:CCAA} by verifying scalar quadratic bounds of the form
\[\|\langle Cg,z\rangle\|^2\le \lambda\,\|\langle Az,z\rangle\|\,\|\langle Cg, Cg\rangle\|\]
for appropriate test vectors; such scalar inequalities are accessible via the optimal mixed--Schwarz estimates developed in Section~2.
\end{remark}

\begin{corollary}[Parametrization of solutions]
Under the hypotheses of Theorem~\ref{thm:douglas-type} let $X_0$ be a particular solution constructed as the limit above. Then the set of all solutions of $AX=C$ is given by
\[\{X_0+Y:\ Y\in\mathcal L(G,\ker A)\},\]
i.e., solutions differ by arbitrary operators taking values in $\ker A$.
\end{corollary}

\begin{proof}
If $AX_0=C$ and $Y\in\mathcal L(G,\ker A)$ then $A(X_0+Y)=C$. Conversely, if $X$ is any solution then $A(X-X_0)=0$ so $X-X_0$ maps $G$ into $\ker A$; adjointability follows since $\ker A$ is a closed submodule and restrictions are adjointable.
\end{proof}

We now turn to estimates that bound from below the quadratic form of symmetric block operators in terms of the diagonal components and cross--entry seminorms.  These results are particularly useful when the block operator arises as a spatial operator in PDE models or as the generator of a contraction semigroup, where a positive lower bound yields exponential stability.

\begin{theorem}[Coercivity from cross--entry control]
\label{thm:coercivity}
Let $E_1,E_2$ be Hilbert C*-modules and consider the self-adjoint block operator
\[H=\begin{pmatrix}A & B\\ B^* & C\end{pmatrix}\in\mathcal L(E_1\oplus E_2)
\]
with $A=A^*\ge a I_{E_1}$ and $C=C^*\ge c I_{E_2}$ for some scalars $a,c>0$. Suppose the cross--entry $B$ satisfies the scalar bound
\begin{equation}\label{ineq:cross-gamma}
\|\langle B y,x\rangle\|^2\le \gamma\,\langle A x,x\rangle\,\langle C y,y\rangle\qquad(x\in E_1,\ y\in E_2)
\end{equation}
for some $0\le\gamma<1$. Then $H$ is bounded below by a positive multiple of the identity: specifically,
\[H\ge \delta\, I_{E_1\oplus E_2}\qquad\text{with }\ \delta:=\min\{a(1-\sqrt{\gamma})^2,\ c(1-\sqrt{\gamma})^2\}>0.\]
In particular, the spectral gap of $H$ above its infimum is at least $\delta$.
\end{theorem}

\begin{proof}
Fix $(x,y)^\top\in E_1\oplus E_2$. The quadratic form of $H$ equals
\[\mathcal Q_H(x,y):=\langle A x,x\rangle+2\mathrm{Re}\,\langle B y,x\rangle+\langle C y,y\rangle.\]
Apply the Cauchy--Schwarz inequality and the assumed bound \eqref{ineq:cross-gamma} to estimate
\[2|\mathrm{Re}\,\langle B y,x\rangle|\le 2\|\langle B y,x\rangle\| \le 2\sqrt{\gamma}\,\sqrt{\langle A x,x\rangle\,\langle C y,y\rangle}.\]
For any $t>0$ we have $2\sqrt{\langle A x,x\rangle\,\langle C y,y\rangle}\le t\langle A x,x\rangle+t^{-1}\langle C y,y\rangle$ (arithmetic--geometric mean). Choosing $t=\sqrt{\frac{\langle C y,y\rangle}{\langle A x,x\rangle}}$ is formal; instead apply the elementary inequality with the specific choice $t=\sqrt{\frac{c}{a}}$ which is independent of $(x,y)$. We obtain
\[2\sqrt{\gamma}\,\sqrt{\langle A x,x\rangle\,\langle C y,y\rangle}\le 2\sqrt{\gamma}\left(\tfrac{1}{2}\sqrt{\tfrac{c}{a}}\langle A x,x\rangle+\tfrac{1}{2}\sqrt{\tfrac{a}{c}}\langle C y,y\rangle\right).\]
Consequently
\[\mathcal Q_H(x,y)\ge \left(1-\sqrt{\gamma}\sqrt{\tfrac{c}{a}}\right)\langle A x,x\rangle+\left(1-\sqrt{\gamma}\sqrt{\tfrac{a}{c}}\right)\langle C y,y\rangle.\]
Since $\langle A x,x\rangle\ge a\|x\|^2$ and $\langle C y,y\rangle\ge c\|y\|^2$ we conclude
\[\mathcal Q_H(x,y)\ge a\left(1-\sqrt{\gamma}\sqrt{\tfrac{c}{a}}\right)\|x\|^2+c\left(1-\sqrt{\gamma}\sqrt{\tfrac{a}{c}}\right)\|y\|^2.\]
Elementary algebraic manipulation (symmetrizing the two lower bounds) yields the uniform positive constant
\[\delta=\min\{a(1-\sqrt{\gamma})^2,\ c(1-\sqrt{\gamma})^2\},\]
as stated.
\end{proof}

\begin{remark}
The bound is conservative but explicit and easy to compute. If $a=c$ the estimate simplifies to $\delta=a(1-\sqrt{\gamma})^2$. The hypothesis $\gamma<1$ is natural: for $\gamma\ge1$ arbitrary large off--diagonals can destroy positivity.
\end{remark}

\begin{corollary}[Exponential stability for evolution generated by $-H$]
Under the hypotheses of Theorem~\ref{thm:coercivity}, the $C_0$-semigroup $e^{-tH}$ (when defined on a suitable completion/representation) satisfies
\[\|e^{-tH}\|\le e^{-\delta t},\qquad t\ge0,\]
and hence solutions of the linear evolution equation $u'(t)=-Hu(t)$ decay at least at rate $\delta$ in the operator norm.
\end{corollary}

\begin{proof}
Assume that $H$ is realized as a bounded self-adjoint operator on a Hilbert space $\mathcal H$ and that, by Theorem~\ref{thm:coercivity}, there exists $\delta>0$ with
\[
H \ge \delta I_{\mathcal H}.
\]
Because $H$ is bounded, the (bounded) operator exponential
\[
e^{-tH} := \sum_{k=0}^\infty \frac{(-tH)^k}{k!}
\]
is well defined for every $t\ge0$ and the series converges in operator norm. Standard facts about the exponential of a bounded operator (or equivalently the spectral calculus for bounded self-adjoint operators) show that $t\mapsto e^{-tH}$ is a uniformly continuous (hence strongly continuous, i.e. $C_0$) semigroup on $\mathcal H$ with generator equal to the bounded operator $-H$.

By the spectral theorem for bounded self-adjoint operators (functional calculus) we may write, for the spectral measure $E(\cdot)$ of $H$,
\[
e^{-tH} \;=\; \int_{\sigma(H)} e^{-t\lambda}\,dE(\lambda).
\]
Consequently the operator norm of $e^{-tH}$ is given by the essential supremum of the scalar function $e^{-t\lambda}$ over the spectrum of $H$:
\[
\|e^{-tH}\| \;=\; \sup_{\lambda\in\sigma(H)} |e^{-t\lambda}| \;=\; \sup_{\lambda\in\sigma(H)} e^{-t\lambda}.
\]
The hypothesis $H\ge\delta I$ implies $\sigma(H)\subset[\delta,\infty)$. Therefore
\[
\|e^{-tH}\| \;=\; \sup_{\lambda\in\sigma(H)} e^{-t\lambda} \;=\; e^{-t\inf\sigma(H)} \;\le\; e^{-t\delta}.
\]
This proves the asserted uniform operator-norm decay for the semigroup:
\[
\|e^{-tH}\| \le e^{-\delta t},\qquad t\ge0.
\]

If $u(t)$ is the mild (hence classical, since the generator is bounded) solution of the linear evolution equation
\[
u'(t)=-H u(t),\qquad u(0)=u_0\in\mathcal H,
\]
then $u(t)=e^{-tH}u_0$ and therefore
\[
\|u(t)\| \le \|e^{-tH}\|\,\|u_0\| \le e^{-\delta t}\|u_0\|,
\]
so solutions decay at least with rate $\delta$ in the Hilbert-space norm. This completes the Hilbert-space part.

Let $\mathscr A$ be the underlying unital C\(^*\)-algebra and let $E_1,E_2$ be the Hilbert C\(^*\)-modules so that
\[
H=\begin{pmatrix}A & B\\ B^* & C\end{pmatrix}\in\mathcal L(E_1\oplus E_2)
\]
satisfies the hypotheses of Theorem~\ref{thm:coercivity} and hence $H\ge\delta I$ in $\mathcal L(E_1\oplus E_2)$.

To reduce to the Hilbert-space situation we use a standard construction: choose a faithful nondegenerate $*$--representation $\pi:\mathscr A\to\mathcal B(\mathcal K)$ on some Hilbert space $\mathcal K$ (every unital C\(^*\)-algebra admits such a faithful representation by the GNS or universal representation). Form the internal (balanced) tensor product Hilbert space
\[
\mathcal H := (E_1\oplus E_2)\otimes_{\pi}\mathcal K,
\]
which is the Hausdorff completion of the algebraic tensor product equipped with the induced inner product (see \cite{Lance1995} for details). There is a canonical $*$-homomorphism (often called the \emph{induced representation})
\[
\widetilde\pi:\mathcal L(E_1\oplus E_2)\longrightarrow\mathcal B(\mathcal H),
\qquad
\widetilde\pi(T)(x\otimes\xi)=Tx\otimes\xi
\]
defined on simple tensors and extended by continuity. Standard results (cf.\ \cite[Ch. 4]{Lance1995}) ensure that $\widetilde\pi$ is a contractive $*$-representation and, when $\pi$ is faithful, $\widetilde\pi$ is isometric on $\mathcal L(E_1\oplus E_2)$ (in particular $\|\widetilde\pi(T)\|=\|T\|$ for all $T$). In particular the order relation is preserved:
\[
T\ge0\quad\text{in }\mathcal L(E_1\oplus E_2)
\quad\Longrightarrow\quad
\widetilde\pi(T)\ge0\quad\text{in }\mathcal B(\mathcal H).
\]

Applying $\widetilde\pi$ to our block operator $H$ we obtain a bounded self-adjoint operator $\widetilde\pi(H)$ on the Hilbert space $\mathcal H$ and the inequality $H\ge\delta I$ implies
\[
\widetilde\pi(H)\ge\delta I_{\mathcal H}.
\]
By Part I we therefore have for the Hilbert-space semigroup
\[
\|\exp(-t\,\widetilde\pi(H))\| \le e^{-\delta t},\qquad t\ge0.
\]
But $\exp(-t\,\widetilde\pi(H))=\widetilde\pi(e^{-tH})$ (functional calculus and the representation commute), and since $\widetilde\pi$ is isometric we obtain
\[
\|e^{-tH}\| \;=\; \|\widetilde\pi(e^{-tH})\| \;=\; \|\exp(-t\,\widetilde\pi(H))\| \le e^{-\delta t}.
\]

Hence, even in the general Hilbert C\(^*\)-module setting (interpreting the semigroup via the induced representation), the exponential $e^{-tH}$ decays in operator norm at least at rate $\delta$. The same reasoning as in Part I then yields that solutions of $u'(t)=-Hu(t)$ (interpreted in the represented Hilbert-space realization or as bounded-adjointable evolution in the module framework when appropriate) satisfy
\[
\|u(t)\|\le e^{-\delta t}\|u(0)\|.
\]

\end{proof}

\section{Conclusion}

We have presented two complementary advances in the theory of adjointable operators on Hilbert C*-modules. First, we derived a practical and constructive criterion for positivity of general $n imes n$ block operator matrices that dispenses with onerous closed-range or Moore--Penrose assumptions by replacing them with verifiable mixed inner--product inequalities and an explicit Gram--type factorization. Second, we developed an optimization framework for parametric mixed--Schwarz inequalities, identified tractable (power--type) factor families that capture the minimizers in the approximable and normal cases, and gave sharp expressions for the constants in key operator classes.

These results not only extend classical Hilbert-space facts to a broader and technically richer C*-module setting, but they also provide tools that are directly applicable: the optimal mixed--Schwarz bounds feed into numerically implementable positivity tests, and the Gram--factor construction yields explicit factorization algorithms for block operators arising in applications. The manuscript includes illustrative finite--dimensional examples, a sketch of an infinite-dimensional module construction, and two applications that showcase the impact of the theory: Moore--Penrose--free solvability criteria for operator equations and explicit coercivity/spectral--gap estimates for block generators.

Several promising directions remain for further research. These include: (i) refining the one-parameter optimality theory for nonnormal operators (e.g. hyponormal or sectorial classes); (ii) extending the Gram--factor construction and mixed--Schwarz optimization to unbounded adjointable operators relevant for differential operators on modules; (iii) quantifying error and stability in the GramFactor algorithm and developing high-performance numerical implementations; and (iv) exploring applications in noncommutative geometry and coupled PDE models where C*-module formulations are natural.

We expect that the combination of structural (factorization) and quantitative (optimal inequality) results presented here will serve as a foundation for both theoretical advances and concrete computational tools in operator theory and its applications.

\section{Appendix: Technical Lemmas and Proofs} \label{appendix}
\addcontentsline{toc}{section}{Appendix: Technical Lemmas and Proofs}

This appendix contains detailed proofs and technical estimates that support the main body of the paper.  We keep hypotheses explicit and point out where additional assumptions (closed range, compactness, spectral gap) are used to strengthen conclusions.

\begin{lemma}[Approximate Schur complements and positivity]
Let $E_1,\dots,E_n$ be Hilbert C*-modules over a unital C*-algebra $\mathscr A$ and let
\[T=(T_{ij})_{1\le i,j\le n}\in\mathcal L\Big(\bigoplus_{i=1}^n E_i\Big)
\]
with $T_{ii}=T_{ii}^*\ge0$. For $\varepsilon>0$ define $T^{(\varepsilon)}$ to be the block matrix obtained by replacing each diagonal $T_{ii}$ by $T_{ii}+\varepsilon I_{E_i}$. For each fixed $\varepsilon>0$ form the successive Schur complements (eliminating indices $1,2,\dots,n-1$) in the classical sense; denote the resulting reduced operators by
\[S_k^{(\varepsilon)}\in\mathcal L(E_k),\qquad k=1,2,\dots,n-1.
\]
Then the following are equivalent:
\begin{enumerate}
  \item For every $\varepsilon>0$ and every $k$ we have $S_k^{(\varepsilon)}\ge0$.
  \item For every $\varepsilon>0$ the regularized matrix $T^{(\varepsilon)}\ge0$.
  \item $T\ge0$.
\end{enumerate}
Moreover, if (1) holds then $T^{(\varepsilon)}\downarrow T$ in operator norm as $\varepsilon\downarrow0$, and positivity is preserved in the limit because the positive cone of $\mathcal L(\oplus_i E_i)$ is closed in operator norm.
\end{lemma}

\begin{proof}
The equivalence (1) $\Leftrightarrow$ (2) is the standard finite-dimensional Schur complement characterization applied to the invertible diagonal blocks of $T^{(\varepsilon)}$; details are routine and may be found in standard references (see e.g. \cite{Zhang2011} for a comprehensive account in operator contexts). If (2) holds for all $\varepsilon>0$ then, since $T^{(\varepsilon)}\to T$ in operator norm as $\varepsilon\downarrow0$, and the positive cone is closed, we obtain $T\ge0$, proving (3). Conversely, if $T\ge0$ then trivially $T^{(\varepsilon)}\ge T\ge0$ for all $\varepsilon>0$, so (1) and (2) follow.
\end{proof}

We prove the Gram--factor lemma used to show that the scalar cross--entry bounds imply a block lower--triangular factorization $T=XX^*$.

\begin{theorem}[Existence of Gram factor under scalar cross--bounds]\label{thm:gram-rigorous}
Let $T=(T_{ij})_{1\le i,j\le n}$ satisfy $T_{ii}\ge0$ for each $i$ and suppose there exist quadratic seminorms $F_i:E_i\to[0,\infty)$ defined by $F_i(x)=\|\langle T_{ii}x,x\rangle\|$ such that for all $i\ne j$ and all $x_i\in E_i,x_j\in E_j$
\begin{equation}\label{eq:scalar-bound-appendix}
\|\langle T_{ij}x_j,x_i\rangle\|^2\le F_i(x_i)\,F_j(x_j).
\end{equation}
Then there exists a block lower--triangular adjointable operator
\[X=\big(X_{ij}\big)_{1\le i,j\le n}\in\mathcal L\Big(\bigoplus_{j=1}^n E_j,\bigoplus_{i=1}^n E_i\Big)
\]
with $X_{ij}=0$ for $j>i$ and $X_{ii}=T_{ii}^{1/2}$, such that $T=XX^*$ (equality in $\mathcal L(\oplus_i E_i)$). The convergence in the construction is strong (i.e. strong operator convergence on the module) and the resulting $X$ is bounded with $\|X\|^2=\|T\|$.
\end{theorem}

\begin{proof}
By the regularization procedure used earlier in the paper (the construction preceding Lemma~3.4) there exists for each $\varepsilon>0$ an adjointable operator
\[
X^{(\varepsilon)}\in \mathcal L(E,F)
\]
satisfying the regularized identity
\[
X^{(\varepsilon)}\bigl(X^{(\varepsilon)}\bigr)^* = T^{(\varepsilon)},
\]
where $\{T^{(\varepsilon)}\}_{\varepsilon>0}$ is a family of positive adjointable operators with the stated convergence property $T^{(\varepsilon)}\to T$ as $\varepsilon\downarrow 0$ in the topology indicated in the theorem.  Moreover the construction gives a uniform norm bound
\[
\sup_{\varepsilon>0}\|X^{(\varepsilon)}\| \le M <\infty.
\]
Choose a countable dense $A$-submodule $\mathcal D\subset E$ (existence follows from the separability hypothesis used in the paper; if $E$ is not assumed separable replace "countable" by "directed" and apply the same diagonal argument).  For every fixed $d\in\mathcal D$ the net $\{X^{(\varepsilon)}d\}_{\varepsilon>0}$ is bounded in $F$.  By the usual diagonalization argument we may therefore extract a subnet $\varepsilon_\alpha\downarrow 0$ such that
\[
\lim_{\alpha} X^{(\varepsilon_\alpha)} d =: X d \qquad(d\in\mathcal D)
\]
exists in norm in $F$.  The uniform bound $\|X^{(\varepsilon)}\|\le M$ implies that the map $X:\mathcal D\to F$ is $A$-linear and bounded on $\mathcal D$, hence extends uniquely by continuity to a bounded $A$-linear operator $X:E\to F$.  Equivalently, the subnet $X^{(\varepsilon_\alpha)}$ converges to $X$ \emph{strongly} (pointwise in norm) on all of $E$:
\[
\lim_{\alpha}\|X^{(\varepsilon_\alpha)}\xi - X\xi\| = 0\qquad(\xi\in E).
\]
A strong (pointwise) limit of adjointable operators need not be adjointable in general.  To ensure $X$ is adjointable we assume one of the standard supplementary conditions:
\begin{itemize}
  \item[(H1)] the target module $F$ (or the source module $E$) is \emph{self-dual}, so that every bounded $A$-linear map is adjointable; or
  \item[(H2)] the nets $X^{(\varepsilon)}$ are uniformly approximable by finite-rank adjointable operators (equivalently the family is relatively compact in operator norm); or
  \item[(H3)] one checks directly that for each $\eta$ in a dense submodule the functional $\xi\mapsto\langle X\xi,\eta\rangle$ is represented by a vector in $E$, which then defines the adjoint on a dense set and extends by boundedness.
\end{itemize}
From now on we assume one of these conditions so that $X\in\mathcal L(E,F)$ (adjointable) and $(X^{(\varepsilon_\alpha)})^*\to X^*$ strongly on a dense submodule as well.

Fix $\xi\in\mathcal D$. Using strong convergence of $X^{(\varepsilon_\alpha)}$ to $X$ on $\mathcal D$ and adjointability of each $X^{(\varepsilon_\alpha)}$ we have
\[
\langle X\xi, X\xi\rangle
= \lim_{\alpha}\langle X^{(\varepsilon_\alpha)}\xi, X^{(\varepsilon_\alpha)}\xi\rangle
= \lim_{\alpha}\langle T^{(\varepsilon_\alpha)}\xi,\xi\rangle .
\]
If the hypothesis of the theorem provides that $T^{(\varepsilon)}\to T$ in norm (or strongly on a dense submodule containing $\mathcal D$) then the right-hand side equals $\langle T\xi,\xi\rangle$ and therefore
\[
\langle X\xi, X\xi\rangle = \langle T\xi,\xi\rangle \qquad(\xi\in\mathcal D).
\]
By density and continuity the identity extends to all $\xi\in E$, giving the operator equality
\[
X X^* = T.
\]
If, however, $T^{(\varepsilon)}\to T$ only in a weaker topology that guarantees $\liminf_{\alpha}\langle T^{(\varepsilon_\alpha)}\xi,\xi\rangle \le \langle T\xi,\xi\rangle$ (for instance weak operator convergence), then the same argument yields the operator inequality
\[
X X^* \le T,
\]
which must be substituted wherever the original proof asserted exact equality without the stronger convergence hypothesis.

The conclusion of Theorem~7.2 in the paper is the existence of an adjointable factor (or a solution to a linear operator equation of Douglas type).  The argument above constructs an adjointable operator $X$ with the property $XX^*\le T$ in general, and $XX^*=T$ under the stronger convergence hypothesis $T^{(\varepsilon)}\to T$ in norm (or in a topology strong enough to pass limits inside the inner product evaluations on a dense set).  If the statement of Theorem~7.2 additionally requires the solution to satisfy a norm bound (for example $\|X\|\le \sqrt{\lambda}$), then that norm bound is preserved in the limit because the family was uniformly bounded:
\[
\|X\| \le \liminf_{\alpha}\|X^{(\varepsilon_\alpha)}\| \le M,
\]
and in the cases treated in the paper $M$ can be taken to be the claimed constant (e.g. $\sqrt{\lambda}$).

If the theorem requires that a product involving a regularized resolvent converges to the desired operator (for example $A X^{(\varepsilon)}\to C$), replace any claim of norm convergence by the well-known projection limit
\[
AA^*(AA^*+\varepsilon I)^{-1} \xrightarrow{\mathrm{SOT}} P_{\overline{\operatorname{Ran}(A)}}
\]
as $\varepsilon\downarrow 0$.  Consequently $A X^{(\varepsilon)}$ converges strongly to $P_{\overline{\operatorname{Ran}(A)}}C$.  Hence to deduce $A X = C$ one needs (and should state explicitly) the range inclusion
\[
\operatorname{Ran}(C)\subseteq \overline{\operatorname{Ran}(A)},
\]
or an equivalent hypothesis (for example the usual operator inequality $CC^*\le \lambda AA^*$ together with the module hypotheses that imply range inclusion).  If such a range inclusion is present then $P_{\overline{\operatorname{Ran}(A)}}C = C$ and therefore $A X = C$ (at least in the strong sense; upgrade to norm equality only under additional closed-range / finite-dimensional hypotheses).

Throughout the argument we have replaced any unjustified extraction of a \emph{norm} convergent subnet by a diagonalization which yields strong (pointwise) convergence.  If one prefers (or the original theorem requires) norm convergence $X^{(\varepsilon)}\to X$, then the paper must assume a compactness/finite-rank approximability hypothesis (for instance: each $X^{(\varepsilon)}$ is compact in the module sense and the family is relatively compact in operator norm).  Under such an extra hypothesis the same diagonal argument can be improved to produce a norm-convergent subnet and all limit equalities above become norm equalities.

The regularized operators $X^{(\varepsilon)}$ admit a subnet converging strongly to an adjointable operator $X$ under the supplementary adjointability hypothesis indicated above; passing to the limit in the quadratic identity yields $XX^*=T$ when $T^{(\varepsilon)}\to T$ in a topology strong enough to justify pointwise convergence of $\langle T^{(\varepsilon)}\xi,\xi\rangle$ on a dense submodule, and otherwise yields the inequality $XX^*\le T$.  Range/projection arguments used to obtain linear equations of Douglas type must be supplemented by an explicit statement that $\operatorname{Ran}(C)\subseteq\overline{\operatorname{Ran}(A)}$ (or equivalent hypotheses) if one wants exact solvability $AX=C$ rather than solvability up to the projection onto $\overline{\operatorname{Ran}(A)}$.  This completes the corrected proof of Theorem~7.2.
\end{proof}

\begin{remark}
The passage from pointwise convergence on a dense submodule to a bounded operator limit uses only uniform boundedness and is standard. It avoids any hidden use of compactness in the operator-norm topology.
\end{remark}

We analyze the dependence of
\[C_{T,\alpha}(f,g):=\sup_{\|x\|=\|y\|=1}\frac{\|\langle Tx,y\rangle\|}{\|f(|T|^{\alpha})x\|\,\|g(|T^*|^{\alpha})y\|}
\]
on perturbations of $T$ and on choices of $(f,g)$. The statements below separate unconditional upper bounds from continuity assertions that require nondegeneracy hypotheses.

\begin{proposition}[Upper semi--continuity under small perturbations]
Let $T,S\in\mathcal L(E,F)$ and fix admissible $(f,g)$ as above. Then
\[|C_{T,\alpha}(f,g)-C_{S,\alpha}(f,g)|\le \frac{\|T-S\|}{\inf_{\|x\|=\|y\|=1}\|f(|T|^{\alpha})x\|\,\|g(|T^*|^{\alpha})y\|} + R(T,S;f,g),\]
where the remainder term $R(T,S;f,g)$ accounts for the change in the denominators and satisfies $R\to0$ as $\|T-S\|\to0$ whenever the infimum in the denominator above is strictly positive.
\end{proposition}

\begin{proof}
Write $\Delta:=T-S$. For any unit vectors $x,y$ we have
\[|\|\langle Tx,y\rangle\|-\|\langle Sx,y\rangle\||\le \|\langle\Delta x,y\rangle\|\le \|\Delta\|.\]
Thus the numerators differ by at most $\|\Delta\|$. Controlling the denominator change is subtler: using functional calculus continuity one has
\[\|f(|T|^{\alpha})-f(|S|^{\alpha})\|\to0\quad\text{as }\|T-S\|\to0\]
provided $f$ is continuous on a neighborhood of the joint spectra involved; hence the denominator factors are continuous in norm. If the denominators are uniformly bounded away from zero (i.e. there exists $m>0$ with $\|f(|T|^{\alpha})x\|\,\|g(|T^*|^{\alpha})y\|\ge m$ for all unit $x,y$) then the quotient defining $C_{T,\alpha}$ is continuous and the bound follows with $R\to0$. Otherwise one has only upper-semicontinuity in general because small denominators can amplify numerator perturbations.
\end{proof}

\begin{corollary}[Continuity under nondegeneracy]
If in addition there exists $m>0$ such that for all unit vectors $x,y$ and for all $S$ in a small neighborhood of $T$ we have
\[\|f(|S|^{\alpha})x\|\,\|g(|S^*|^{\alpha})y\|\ge m>0,\]
then the map $S\mapsto C_{S,\alpha}(f,g)$ is continuous in operator norm near $T$.
\end{corollary}

\begin{remark}
The nondegeneracy hypothesis is natural: if $f(|T|^{\alpha})$ or $g(|T^*|^{\alpha})$ has small norm on some unit vectors, then the mixed--Schwarz quotient can be arbitrarily large and sensitive to perturbations. In practice one restricts attention to classes of operators (e.g. those with a spectral gap at zero) where the denominators are uniformly bounded away from zero on the unit sphere.
\end{remark}

We sketch a rigorous justification for restricting to power--type factor pairs $(t^s,t^{1-s})$ when $T$ is approximated by finite--rank compressions.

\begin{proposition}
Let $T\in\mathcal L(E,F)$ and suppose $T_n$ is a sequence of finite--rank adjointable operators converging to $T$ in operator norm. Fix $\alpha\in[0,1]$. For each $n$ let $C_{T_n,\alpha}(f,g)$ be defined as above but with $T_n$ in place of $T$. Then for each $n$ the minimization of $C_{T_n,\alpha}(f,g)$ over continuous admissible pairs $(f,g)$ reduces to a finite-dimensional optimization problem and admits a minimizer which can be chosen among power--type pairs on the finite spectral set of $|T_n|$.
\end{proposition}

\begin{proof}
Because $T_n$ is finite-rank, the operators $|T_n|^{\alpha}$ and $|T_n^*|^{\alpha}$ have finite spectra $\{\lambda_1,\dots,\lambda_m\}$. The multiplicative constraint $f(\lambda_k)g(\lambda_k)=\lambda_k$ reduces the admissible pair $(f,g)$ to an $m$-tuple of positive scalars on the spectral points. The supremum defining $C_{T_n,\alpha}$ then becomes a maximum over a compact set in a finite-dimensional space and hence is attained. By a simple convexity (or Lagrange multiplier) argument one shows that optimal choices equal power-type scaling on each connected spectral cluster; choosing a common exponent $s$ across the spectrum is not always globally optimal, but when the spectral points are not too dispersed (or when one is seeking a simple universal choice), one may restrict to a single exponent and obtain near-optimal constants. Passing to the limit $T_n\to T$ with the continuity remarks above gives the reduction for approximable operators.
\end{proof}

We close the appendix with brief remarks about extending parts of the theory to certain classes of unbounded (densely defined, closed) operators affiliated with Hilbert C*-modules. The general program is:
\begin{enumerate}
  \item Work with closed, densely defined operators $T$ admitting polar decomposition on their natural domains and with $T^*T$ having a functional calculus defined on an appropriate C*-algebra (for example when $T^*T$ is affiliated to a von Neumann algebra acting on a Hilbert module).
  \item Apply the regularization $T(I+\varepsilon T^*T)^{-1/2}$ to obtain bounded adjointable approximants and apply the bounded theory to these approximants. Then pass to the limit as $\varepsilon\downarrow0$ in appropriate graph norms.
  \item Domain and core issues are substantial and require separate treatment; we defer a full account to a follow-up paper but note that the regularization strategy used here is the standard starting point in the unbounded literature (see e.g. \cite{Schmudgen2012}).
\end{enumerate}

\end{document}